\documentclass[10pt]{amsart}
\usepackage{graphicx, amsmath, amssymb, amsthm, amsfonts, mathtools, tikz-cd, comment, tikz, caption, subcaption, url, tabularx, enumitem, mathrsfs, amsrefs, appendix, float}

\newtheorem{introthm}{Theorem}

\newtheorem{thm}{Theorem}[section]
\newtheorem{lem}[thm]{Lemma}

\theoremstyle{definition}
\newtheorem{defn}[thm]{Definition}

\title{Uncountable Abelian Group C*-algebras Fail the Lifting Property}
\author{Miles Gould}

\address{Department of Mathematics,
University of Louisiana at Lafayette,
Lafayette, USA}

\email{miles.gould1@louisiana.edu}

\date{\today}

\begin{document}

\begin{abstract}
    In this note, we show that $C^*(G)$ fails the lifting property (LP) for every discrete group $G$ containing an uncountable abelian subgroup. Consequently, for every uncountable discrete abelian group $G$, the nuclear algebra $C^*(G)$ has the local lifting property (LLP) but fails the LP. This shows that the LP and the LLP do not coincide for full discrete group $C^*$-algebras.
\end{abstract}

\maketitle

\section{Introduction}

The Choi-Effros theorem \cite{CE76} is a fundamental result in $C^*$-algebra theory. Among other things, it says that for a separable, nuclear $C^*$-algebra $A$, every cpc map $\phi:A\rightarrow B/J$ lifts to a cpc map $\theta:A\rightarrow B$. For a general $C^*$-algebra $A$, we call this property the \textit{lifting property} (LP). The Choi-Effros theorem also shows that for every nuclear $C^*$-algebra $A$, every cpc map $\phi: A\rightarrow B/J$, and every finite-dimensional operator subsystem $E$ of $A$, there is a cpc lift $\theta: E\rightarrow B$ of $\phi|_E$.
For a general $C^*$-algebra $A$, we call this the \textit{local lifting property} (LLP).

It is well-known that separability cannot be omitted for the LP claim, even in the abelian case. Indeed, for $\ell^\infty/c_0$, the identity map $\mathrm{id}:\ell^\infty/c_0\rightarrow \ell^\infty/c_0$ does not lift to a cpc map $\theta:\ell^\infty/c_0\rightarrow \ell^\infty$ \cite[\S 3]{OZA}. In this case, the failure of the LP boils down to a failure of the countable chain condition (ccc), that is, $\ell^\infty/c_0$ admits an uncountable family of pairwise orthogonal positive contractions.

However, for discrete, amenable $G$, the full group $C^*$-algebra $C^*(G)=C^*_r(G)$. By \cite[2.5.3]{BROZ}, $C^*_r(G)$ admits a canonical faithful trace, so satisfies the ccc, evading this obstruction. As such, it was posed by Fournier-Facio and Willett \cite[1.13]{WIL} whether LP and LLP are equivalent for full group $C^*$-algebras, particularly whether $C^*(G)$ has the LP for an uncountable free abelian group $G$. In 2004 \cite[\S 3]{OZA}, Ozawa suspected that $C^*(F_\kappa)$ fails the lifting property for uncountable $\kappa$, and this remains open, so the problem is difficult even in the non-nuclear case. We provide a partial negative answer to \cite[1.13]{WIL}, namely that $C^*(G)$ has the LLP yet fails the LP for every uncountable discrete abelian group $G$.

\begin{introthm}\label{disc}[Theorem \ref{disc2}]
    For every discrete group $G$ admitting an uncountable abelian subgroup, $C^*(G)$ fails the lifting property. In particular, $C^*(G)$ fails the lifting property for every uncountable discrete abelian group $G$.
\end{introthm}

To prove theorem \ref{disc}, we show that $C(2^{\aleph_1})$ is a ucp retract of $C^*(G)$ for every discrete group $G$ admitting an uncountable abelian subgroup. As such, we first prove that $C(2^\kappa)$ fails the LP for every uncountable $\kappa$. In particular, we prove the following theorem.

\begin{introthm}\label{LP}[Theorem \ref{LP2}]
    For every compact Hausdorff space $X$ which
    \begin{enumerate}
        \item is totally disconnected,
        \item is not extremally disconnected,
        \item has no $G_\delta$ points,
    \end{enumerate}
    $C(X)$ fails the lifting property. In particular, $C^*(\bigoplus_\mathfrak{\kappa}\mathbb{Z}_2)=C(2^\kappa)$ fails the lifting property for every uncountable $\kappa$.
\end{introthm}

\subsection{Acknowledgements}

This paper arose from a question during private correspondence between Rufus Willett and Ilijas Farah. I thank Farah for including me in this correspondence and for his guidance on this problem. Moreover, the key insight into its solution came from an interaction between Rufus Willett, Piotr Koszmider, and me at the Logic and $C^*$-algebras conference (IlijasFest) in Cetraro, Italy. As such, I thank all parties involved in facilitating this interaction. In particular, after attending Koszmider's talk on the contents of \cite{KOS}, Willett suggested to me that the results therein could be of use. This insight was precisely the key to finding the first example, so I am particularly grateful to him for it. ChatGPT was used substantially in the writing of this note, both in mathematical content and proofreading. I wrote and checked the final version of every proof appearing, and assume responsibility for their validity.

\section{Proof of Theorem \ref{LP}}

Our main target will be a nonliftable ucp map into $\mathscr{Q}(H):=\mathscr{B}(H)/\mathscr{K}(H)$. Throughout, we will often use quotients, and use subscripts to distinguish them. For example, $\pi_H:\mathscr{B}(H)\rightarrow\mathscr{Q}(H)$ and $\pi_J:B\rightarrow B/J$. This abuse will be clear from the context.

\begin{lem}\label{exp}
    Let $A$ be an abelian $C^*$-algebra. Let $\iota:A\hookrightarrow\mathscr{Q}(H)$ be an embedding onto a masa of $\mathscr{Q}(H)$. If $\iota$ lifts to a ucp map $\theta:A\rightarrow\mathscr{B}(H)$, then there exists a von Neumann algebra $N$ and a conditional expectation $E:N\rightarrow A$.
\end{lem}

\begin{proof}
    Let $(K,\rho,V)$ be a Stinespring dilation of $\theta$. That is, $\rho:A\rightarrow\mathscr{B}(K)$ is a representation and $V:H\rightarrow K$ an isometry such that $\theta(a)=V^*\rho(a)V$ for all $a\in A$. In particular, since $\theta$ lifts $\iota$, $\iota(a)=\pi_H(V^*\rho(a)V)$. Define $P:=VV^*$.
    Then
    \[0=\iota(a^*a)-\iota(a)^*\iota(a)=\pi_H(V^*\rho(a)^*(1-P)\rho(a)V)\]
    Thus $(1-P)\rho(a)V$ is compact, so $(1-P)\rho(a)P$ is as well.
    \begin{align*}
        [\pi_K(\rho(a)),\pi_K(P)] &= \pi_K(\rho(a)P-P\rho(a)) \\
        &= \pi_K((1-P)\rho(a)P+P\rho(a)P-P\rho(a)(1-P)-P\rho(a)P) \\
        &=\pi_K((1-P)\rho(a)P-P\rho(a)(1-P))=0
    \end{align*}
Let $N:=\rho[A]'$. Since $A$ is abelian, $\rho[A]\subseteq N$ and $N$ is a von Neumann algebra. Define $\phi: N\rightarrow\mathscr{Q}(H)$ by $\phi(T):=\pi_H(V^*TV)$. We claim $\phi[N]\subseteq\iota[A]'$. Fix $T\in N$, $a\in A$. Then

\begin{align*}
    \phi(T)\iota(a) &= \pi_H(V^*TV)\pi_H(V^*\rho(a)V) \\
    &= \pi_H(V^*TP\rho(a)V) \\
    &= \pi_H(V^*T\rho(a)PV) \\
    &= \pi_H(V^*T\rho(a)V) \\
    &= \pi_H(V^*\rho(a)TV) \\
    &= \pi_H(V^*P\rho(a)TV) \\
    &= \pi_H(V^*\rho(a)PTV) \\
    &= \pi_H(V^*\rho(a)V)\pi_H(V^*TV) =\iota(a)\phi(T).
\end{align*}
Since $\iota[A]$ is a masa, $\iota[A]'\subseteq\iota[A]$. Define $E:=\rho\circ\iota^{-1}\circ\phi: N\rightarrow \rho[A]$. Then $E$ is ucp and for $a\in A$,
\[E(\rho(a))=\rho(\iota^{-1}(\pi_H(V^*\rho(a)V))=\rho(\iota^{-1}(\iota(a))=\rho(a),\]
so $E$ is a projection onto $\rho[A]$ of norm 1, so $E$ is a conditional expectation \cite[II.6.10.2]{BLA}. Since $\theta$ lifts $\iota$, $\theta$ is injective. Hence $\rho$ must be faithful. Therefore, $A$ is isomorphic to $\rho[A]$ and the conclusion follows.
\end{proof}

The following is the main tool which we use to prove theorem \ref{LP}. It allows us to preserve enough of the LP through $\sigma$-closed forcings.

\begin{lem}\label{forc}
    Let $M$ be a transitive model of ZFC, $\mathbb{P}$ a $\sigma$-closed forcing, and $G$ a $\mathbb{P}$-generic filter. If a unital $C^*$-algebra $A$ has the LP in $M$, then every unital *-homomorphism $\phi: A\rightarrow B/J$ has a ucp lift $\theta: A \rightarrow B$ in $M[G]$.
\end{lem}

\begin{proof}
    Work in $M$. Choose a generating set $S\subseteq A_1$ for $A$, and let
    \[D:=C^*\langle x_s:s\in S\mid \|x_s\|\leq 1\rangle,\]
    that is, $D$ is the universal $C^*$-algebra generated by $S$-many contractions. Let $q: D\rightarrow A$ by $x_s\mapsto s$ for $s\in S$, which is a surjective unital *-homomorphism. Let $\bar{q}:D/\ker(q)\rightarrow A$ be the induced isomorphism. As $A$ has the LP, $\bar{q}^{-1}:A\rightarrow D/\ker(q)$ has a cpc lift $\sigma: A\rightarrow D$. As noted below \cite[C.1]{BROZ}, $\sigma$ can be taken to be ucp. After identifying $D/\ker(q)$ with $A$, we obtain $q\sigma=\bar{q}^{-1}=\mathrm{id}_A$.

    For $F\in[S]^{<\aleph_0}$, define $D_F:=C^*(x_s:s\in F)$. We claim that 
    \[D_F\cong C_F:=C^*\langle y_s:s\in F\mid \|y_s\|\leq 1\rangle\]
    in $M[G]$. Since $F$ is finite, $D_F$ is separable. Since $\mathbb{P}$ is $\sigma$-closed, it adds no new sequences in $D_F$. Moreover, the norm is unchanged, so every Cauchy sequence in $D_F$ in $M[G]$ is Cauchy in $M$, so converges in $M$, and so in $M[G]$. Therefore $D_F$ is complete. Since $M[G]$ adds no new reals, $\mathbb{C}$ is unchanged, so $D_F$ is still a $C^*$-algebra in $M[G]$. Suppose that, in $M[G]$, the norm on $D_F$ were strictly smaller than the universal norm. Then there would exist contractions $\bar{b}:=(b_s:s\in F)$ in a separable $C^*$-algebra $C:=C^*(b_s:s\in F)$ and a *-polynomial $p$ with coefficients in $\mathbb{Q}[i]$ and variables in $\bar{b}$ such that $\|p(\bar{b})\|>\|p(\bar{x})\|$, where $\bar{x}:=(x_s)_{s\in F}$, but this would be coded by a new real, as demonstrated in \cite[\S 2.3]{FTT}. Again, $\mathbb{P}$ adds no new reals, a contradiction. Therefore, $D_F$ is equipped with the universal norm, so our claim follows. As $D$ is the direct limit of $D_F$, it is the direct limit of $C_F$. As noted in \cite[2.3.11]{FA19},
    \[C^*\langle y_s:s\in S\mid \|y_s\|\leq 1\rangle\cong \lim_F C_F \cong D,\]
    so $D$ retains its universal property in $M[G]$.

    Now fix a unital *-homomorphism $\phi : A\rightarrow B/J$ in $M[G]$. Then $\phi q:D\rightarrow B/J$ is also a unital *-homomorphism. For each $s\in S$, fix a lift $b_s\in B$ of $\phi q(x_s)$ with $\|b_s\|\leq 1$. By the universal property of $D$, there is a unital *-homomorphism $\rho: D\rightarrow B$ such that $\pi_J \rho=\phi q$. Thus, $\theta=\rho\sigma: A\rightarrow B$ is ucp and 
    \[\pi_J\theta=\pi_J\rho\sigma=\phi q\sigma=\phi,\]
    so $\theta$ lifts $\phi$.
\end{proof}

Now we prove our prototypical example $C(2^\kappa)$ fails the LP.

\begin{thm}\label{LP2}
    For every compact Hausdorff space $X$ which
    \begin{enumerate}
        \item is totally disconnected,
        \item is not extremally disconnected,
        \item has no $G_\delta$ points,
    \end{enumerate}
    $C(X)$ fails the lifting property. In particular, $C^*(\bigoplus_\mathfrak{\kappa}\mathbb{Z}_2)=C(2^\kappa)$ fails the lifting property for every uncountable $\kappa$.
\end{thm}

\begin{proof}
    Work in a ground model $M$ and let $X$ be as in the assumption. Set $A:=C(X)^M$ and suppose, towards contradiction, that $A$ satisfies the LP in $M$. Let $\kappa:=\max\{\mathfrak{c},\mathrm{dens}(A)\}$, where $\mathrm{dens}(A)$ denotes the density character of $A$, and define the forcing
    \[\mathbb{P}:=\{f: s\rightarrow\kappa: s\in[\aleph_1]^{<\aleph_1}\}\]
    ordered under reverse graph inclusion. This is a standard $\sigma$-closed forcing. A $\mathbb{P}$-generic filter $G$ yields a function $g:=\bigcup G$, which is a surjection from $\aleph_1$ onto $\kappa$. Therefore, $\mathrm{dens}^{M[G]}(A)\leq\mathfrak{c}^{M[G]}=\aleph_1$. That is, in $M[G]$, the continuum hypothesis holds and $\mathrm{dens}(A)\leq\mathfrak{c}$. 
    
    We claim that $A$ remains a $C^*$-algebra in $M[G]$ and that its new spectrum,
    \[X_G:=\mathrm{Spec}^{M[G]}(A),\]
    is a compact Hausdorff space which
    \begin{enumerate}
        \item is totally disconnected,
        \item is not extremally disconnected,
        \item has no $G_\delta$ points,
        \item has weight no greater than $\mathfrak{c}^{M[G]}$.
    \end{enumerate}

    To see that $A$ remains a (unital abelian) $C^*$-algebra in $M[G]$, it suffices to show that $A$ is complete and $\mathbb{C}$ is unchanged, as unital abelian $C^*$-algebras are axiomatizable \cite[2.5.1]{FA21}. Since $\mathbb{P}$ is $\sigma$-closed, it adds no new sequences from the ground model. Because the norm is unchanged, every Cauchy sequence in $A$ belonging to $M[G]$ already belongs to $M$, so converges in $M$, hence in $M[G]$. Moreover, since $M[G]$ adds no new reals, $\mathbb{C}$ is unchanged. The fact that $X_G$ is compact Hausdorff follows by Gelfand duality, as $A\cong C(X_G)^{M[G]}$. We identify these for the remainder. For (1), $X$ is totally disconnected, hence $A$ is generated by projections in $M$, i.e. $A=\overline{\mathrm{span}}(\mathrm{Proj}(A))$. Since the underlying metric space $A$ is unchanged, this density remains true, so $X_G$ is totally disconnected. For (2), since $X$ is not extremally disconnected, $A$ is monotone incomplete in $M$. Fix an increasing net $(a_\lambda)$ in $A_\text{sa}$ with no supremum. This net still has no supremum in $A$ in $M[G]$, so $A$ is monotone incomplete, and $X_G$ is not extremally disconnected. For (3), a point $y$ in a compact Hausdorff space $Y$ is $G_\delta$ if and only if there is $f\in C(Y)_+$ such that $f^{-1}(0)=\{y\}$, equivalently,
    \[C(Y)/\overline{fC(Y)}\cong \mathbb{C}.\]
    Therefore, $y\in Y$ is $G_\delta$ iff for some $f\in C(Y)_+$, $C(Y)/\overline{fC(Y)}\cong \mathbb{C}$. It follows that $X_G$ has a $G_\delta$ point if and only if there is $a\in A_+$ such that $A/\overline{aA}\cong\mathbb{C}$ in $M[G]$. For each $a\in A$, $\overline{aA}$ is unchanged, since $b\in \overline{aA}$ iff there exists $b_n\in A$ with $\|ab_n-b\|\rightarrow0$ and $\mathbb{P}$ adds no new sequences of elements of $A$. Since $X$ has no $G_\delta$ points in $M$, for $a\in A_+$, we have $A/\overline{aA}\not\cong\mathbb{C}$ in $M[G]$. Therefore, $X_G$ has no $G_\delta$ points. Claim (4) follows immediately, since 
    \[w(X_G)=\mathrm{dens}^{M[G]}(C(X_G)^{M[G]})=\mathrm{dens}^{M[G]}(A)\leq\mathfrak{c}^{M[G]}=\aleph_1.\]
    
    By Koszmider \cite[1.1(B)]{KOS}, under the continuum hypothesis, properties $(1)$, $(3)$, and $(4)$ on $X_G$ guarantee an embedding $\iota:A\hookrightarrow\mathscr{Q}(H)$ onto a masa. Since $A$ has the LP in $M$, by lemma \ref{forc}, $\iota$ lifts to ucp $\theta: A\rightarrow\mathscr{B}(H)$ in $M[G]$. By lemma \ref{exp}, there exists a von Neumann algebra $N$ and a conditional expectation $E$ from $N$ to $A$. 
    
    We claim that $X_G$ is extremally disconnected, contradicting $(2)$. It suffices to show that $A$ is monotone complete. Let $(a_\lambda)$ be a bounded, increasing, self-adjoint net in $A$. Then since $A$ is a $C^*$-subalgebra of $N$, and $N$ is von Neumann, $a:=\sup_\lambda a_\lambda\in N$. As $E$ is positive, $E(a)\geq E(a_\lambda)$ for every $\lambda$. If $b\in A$ is an upper bound, then it is an upper bound in $N$, so $a\leq b$, so $E(a)\leq E(b)$. Thus $E(a)\in A$ is the least upper bound of $E(a_\lambda)=a_\lambda$ in $A$. As $(a_\lambda)$ was arbitrary, $A=C(X_G)^{M[G]}$ is monotone complete, a contradiction to (2). Therefore, $\iota$ has no ucp lift in $M[G]$, so $C(X)$ fails the LP in $M$.

\end{proof}

\section{Proof of Theorem \ref{disc}}

We now use the failure of the LP for $C(2^{\aleph_1})$ to prove our main result. To do this, we need the following definition and a few lemmas.

\begin{defn}
    Let $A,B$ be unital $C^*$-algebras. We call $A$ a \textit{ucp retract} of $B$ if there exist ucp maps $\phi_1: A\rightarrow B$, $\phi_2: B\rightarrow A$ such that $\phi_2\phi_1=\mathrm{id}_A$. We note that ucp retracts are transitive.
\end{defn}

\begin{lem}\label{free}
    Every uncountable discrete abelian group $G$ contains a subgroup isomorphic to
    \[\bigoplus_{\aleph_1}\mathbb{Z}\text{ or }\bigoplus_{\aleph_1}\mathbb{Z}_p\]
    for some prime $p$.
\end{lem}

\begin{proof}
    Let $T$ be the torsion subgroup of $G$ and 
    \[T_p:=\{g\in T:p^ng=0\text{ for some }n\geq 1\}\]
    for prime $p$. Because $T=\bigoplus_p T_p$, if $T$ is uncountable, then $T_p$ is uncountable for some $p$. For $n\geq 1$, define 
    \[T_{p,n}:=\{g\in T_p:p^ng=0\}.\]
    Multiplication by $p$ gives a map $T_{p,n+1}\rightarrow T_{p,n}$ with kernel $T_{p,1}$. Consequently, if $T_{p,1}$ is countable, so is $T_p=\bigcup_{n=1}^\infty T_{p,n}$, a contradiction. Therefore $T_{p,1}$ is an uncountable $\mathbb{F}_p$-vector space, so contains a copy of $\bigoplus_{\aleph_1}\mathbb{Z}_p$.

    If $T$ is countable, then $G/T$ is an uncountable torsion-free abelian group, so has uncountable torsion-free rank. Thus $G/T$ contains a copy of $\bigoplus_{\aleph_1}\mathbb{Z}$. By choosing $\mathbb{Z}$-linearly independent lifts in $G$, $G$ also contains $\bigoplus_{\aleph_1}\mathbb{Z}$.
\end{proof}

\begin{lem}\label{Zto2}
    For every infinite cardinal $\kappa$, $C(2^\kappa)$ is a ucp retract of $C^*(H)$ for $H$ isomorphic to
    \[\bigoplus_{\kappa}\mathbb{Z}\text{ or }\bigoplus_{\kappa}\mathbb{Z}_p\]
    for some prime $p$.
\end{lem}

\begin{proof}
    Let $K=\mathbb{T}$ if $H=\bigoplus_\kappa\mathbb{Z}$ and $K=\mathbb{Z}_p$ if $H=\bigoplus_{\kappa}\mathbb{Z}_p$. By Pontryagin duality, it suffices to show that $C(2^\kappa)$ is a ucp retract of $C(K^\kappa)$. Choose distinct $z_0,z_1\in K$. By the normality of $K$, there exists continuous $h_0:K\rightarrow[0,1]$ such that $h_0(z_0)=1$ and $h_0(z_1)=0$. Let $h_1:=1-h_0$, so $h_i(z_j)=\delta_{ij}$. Identify $2^\kappa$ with $\{z_0,z_1\}^\kappa\subseteq K^\kappa$. Let $\phi_2:C(K^\kappa)\rightarrow C(2^\kappa)$ be the restriction map, which is a unital *-homomorphism.

    Let $A$ be the *-subalgebra of $C(2^\kappa)$ consisting of functions depending only on finitely many coordinates in $\kappa$. Since this is a *-subalgebra which separates points in $2^\kappa$, by Stone-Weierstrass, $A$ is dense in $C(2^\kappa)$.

    For $F\in[\kappa]^{<\aleph_0}$, let $\pi_F:2^\kappa\rightarrow 2^F$ be the canonical projection and define
    \[A_F:=\{g\circ\pi_F:g\in C(2^F)\},\quad A:=\bigcup_{F\in [\kappa]^{<\aleph_0}}A_F.\]
    For $f=g\circ\pi_F$, define 
    \[\phi_0(f)(x):=\sum_{s\in 2^F}g(s)\prod_{\alpha\in F}h_{s(\alpha)}(x_\alpha)\]
    for $x=(x_\alpha)\in K^\kappa$. If $F\subseteq F'$, $f=g\circ \pi_F=g'\circ\pi_{F'}$ and $s'\in 2^{F'}$. Then $g'(s')=g(s'|_F)$. Consequently, since $f$ depends only on coordinates in $F$, the sum adds a term for every extension of $s\in 2^F$, so
    \begin{align*}
        \sum_{s'\in 2^{F'}}g'(s')\prod_{a\in F'}h_{s'(\alpha)}(x_\alpha) &= \sum_{s\in 2^F}g(s)\prod_{\alpha\in F}h_{s(\alpha)}(x_\alpha) \sum_{t\in 2^{F'\setminus F}}\prod_{\beta\in F'\setminus F}h_{t(\beta)}(x_\beta) \\
        &= \sum_{s\in 2^F}g(s)\prod_{\alpha\in F}h_{s(\alpha)}(x_\alpha)\prod_{\beta\in F'\setminus F}(h_0(x_\beta)+h_1(x_\beta) \\
        &= \sum_{s\in 2^F}g(s)\prod_{\alpha\in F}h_{s(\alpha)}(x_\alpha).
    \end{align*}
    Thus $\phi_0$ is independent of $F$ so long as $f\in A_F$, so $\phi_0$ is well-defined on $A$. This map is positive, unital, and contractive, since the coefficients $\prod_{\alpha\in F}h_{s(\alpha)}(x_\alpha)$ are nonnegative and sum to $1$. Since $A$ is dense in $C(2^\kappa)$, extend $\phi_0$ to $\phi_1:C(2^\kappa)\rightarrow C(K^\kappa)$. Then $\phi_1$ is unital and positive between abelian $C^*$-algebras, so is ucp.

    For $y\in 2^\kappa$, $f\in A_F$, evaluate $\phi_0(f)$ at $z:=(z_{y(\alpha)})_{\alpha<\kappa}$. This yields
    \begin{align*}
        \phi_0(f)(z) &= \sum_{s\in 2^F}g(s)\prod_{\alpha\in F} h_{s(\alpha)}(z_{y(\alpha)}) \\
        &= \sum_{s\in 2^F}g(s)\prod_{\alpha\in F}\delta_{s(\alpha)y(\alpha)} \\
        &= g\circ\pi_F(y)=f(y).
    \end{align*}
    Therefore $\phi_2\phi_0(f)=f$ for every $f\in A$.
    As $A$ is dense in $C(2^\kappa)$, $\phi_2\phi_1=\mathrm{id}_{C(2^\kappa)}$.
\end{proof}

\begin{comment}

\begin{lem}
    Let $H\leq G$ be discrete groups. Then $C^*(H)$ is a ucp retract of $C^*(G)$.
\end{lem}

\begin{proof}
    Let $\phi_1: C^*(H)\rightarrow C^*(G)$ be the canonical inclusion map. Let $\lambda_{G/H}$ be the quasi-regular representation of $G$ on $\ell^2(G/H)$. Then $\varphi(g)=\langle\lambda_{G/H}(g)\delta_{H},\delta_H\rangle$ is a state and its multiplier $M_\varphi : C^*(G)\rightarrow C^*(G)$ is ucp. Then $M_\varphi(g)=1_{H}(g)g$, so evidently $M_\varphi[C^*(G)]\subseteq C^*(H)$. By identifying $C^*(H)$ with its image under $\phi_1$, we obtain ucp $\phi_2: C^*(G)\rightarrow C^*(H)$ and for every $g\in H$, $\phi_2\phi_1(g)=g$, so $\phi_2\phi_1=\mathrm{id}_{C^*(H)}$.
\end{proof}

\end{comment}

\begin{thm}\label{disc2}
    For every discrete group $G$ admitting an uncountable abelian subgroup, $C^*(G)$ fails the lifting property. In particular, $C^*(G)$ fails the lifting property for every uncountable discrete abelian group $G$.
\end{thm}

\begin{proof}
    Let $G$ be a discrete group admitting an uncountable abelian subgroup. By lemma \ref{free}, $G$ admits a subgroup $H$ isomorphic to $\bigoplus_{\aleph_1}\mathbb{Z}$ or $\bigoplus_{\aleph_1}\mathbb{Z}_p$ for some prime $p$. In either case, by lemma \ref{Zto2}, $C(2^{\aleph_1})$ is a ucp retract of $C^*(H)$. It follows by \cite[2.5.8]{BROZ} that there is a canonical embedding $\iota:C^*(H)\hookrightarrow C^*(G)$. By \cite[2.5.11]{BROZ}, there is a conditional expectation $E:C^*(G)\rightarrow C^*(H)$ such that $E\iota=\mathrm{id}_{C^*(H)}$. Therefore, $C^*(H)$ is a ucp retract of $C^*(G)$, so by transitivity, $C(2^{\aleph_1})$ is a ucp retract of $C^*(G)$.
    
    Suppose, towards contradiction, that $C^*(G)$ satisfies the LP. By the above, there exist ucp maps $\phi_1:C(2^{\aleph_1})\rightarrow C^*(G)$, $\phi_2: C^*(G)\rightarrow C(2^{\aleph_1})$ such that $\phi_2\phi_1=\mathrm{id}_{C(2^{\aleph_1})}$. Fix an arbitrary ucp map $\phi:C(2^{\aleph_1})\rightarrow B/J$. Then $\phi\phi_2:C^*(G)\rightarrow B/J$ is ucp, so admits a ucp lift $\theta: C^*(G)\rightarrow B$. Then $\theta\phi_1:C(2^{\aleph_1})\rightarrow B$ is ucp and $\pi_J \theta\phi_1=\phi\phi_2\phi_1=\phi$. By \cite[13.1.2]{BROZ}, $C(2^{\aleph_1})$ satisfies the LP, a contradiction to theorem \ref{LP}.
\end{proof}

\end{document}